\documentclass[12pt,letterpaper]{amsart}

\usepackage{amsmath,amssymb,amsthm}
\usepackage{fullpage}
\usepackage{verbatim}
\usepackage[noadjust]{cite} 
\usepackage{hyperref}
\usepackage[capitalize, nameinlink]{cleveref}
\usepackage{enumitem}
\usepackage[usenames,dvipsnames]{xcolor} 
\usepackage{tikz}
\usepackage[font=small,labelfont=bf, margin=.5in]{caption}
\usepackage{babel}
\usepackage{parskip, amsfonts, amsmath, changepage, amsthm, setspace, geometry, verbatim, MnSymbol,chemarrow, tikz}

\colorlet{prettygreen}{ForestGreen!60!LimeGreen}

\makeatletter
\def\@defaultbiblabelstyle#1{[#1]}
\makeatother

\usetikzlibrary{calc, arrows.meta}
\tikzset{vtx/.style={circle, fill, inner sep=1.5pt}}
\tikzset{openvtx/.style={circle, draw, inner sep=1.5pt}}

\usepackage{tikz-cd}

\newcommand{\N}{{\mathcal N}}

\newcommand{\Z}{{\mathbb Z}}

\title{Four plane unit vectors generate a $3$-colorable graph}
\author{Katherine Eng}
\address{Department of Mathematics, California State University --- Los
Angeles}
\author{Timothy Harris}
\address{USC Keck School of Medicine, Department of Population and Public Health Sciences}
\email{harristd@usc.edu}
\author{Mike Krebs}
\address{Department of Mathematics, California State University --- Los
Angeles}
\email{mkrebs@calstatela.edu}
\author{Mason Meeks}
\email{mason.meeks@gmail.com}
\author{Claudia Maria Schmidt}
\address{Washington State University}
\email{schmidt.c.maria@gmail.com}

\makeatletter
\let\mytitle\@title
\let\myauthor\@author
\makeatother

\newtheorem{Thm}{Theorem}[section]
\newtheorem{Prop}[Thm]{Proposition}
\newtheorem{Lem}[Thm]{Lemma}
\newtheorem{Cor}[Thm]{Corollary}

\theoremstyle{definition}
\newtheorem{Def}[Thm]{Definition}
 
\theoremstyle{remark}
\newtheorem{Ex}[Thm]{Example}

\begin{document}

\begin{abstract}We show that given an arbitrary set of four plane unit vectors $v_1, v_2, v_3, v_4$, the Cayley graph generated by $\{\pm v_1, \pm v_2, \pm v_3, \pm v_4\}$ is always $3$-colorable.  Indeed, we show that this is a specific case of a much more general result wherein we determine the chromatic number of an arbitrary abelian Cayley graph generated by a set of four elements and their negatives, subject to the constraint that the group of relations between those elements has rank no more than $2$.
\end{abstract}


\maketitle

\section{Introduction}

The well-known chromatic number of the plane problem (a.k.a. Hadwiger-Nelson problem) asks for the smallest number $\chi(\mathbb{R}^2)$ of colors needed to assign every point in $\mathbb{R}^2$ a color such that no two points of distance $1$ from each other receive the same color.  As of the time of writing, it is known that $\chi(\mathbb{R}^2)$ equals either $5$, $6$, or $7$.  The book \cite{Soifer} provides a rich history of this subject.

Recall that a subset $S$ of a group is \emph{symmetric} if $x^{-1}\in S$ whenever $x\in S$.  For a group $G$ with a symmetric subset $S$, the Cayley graph $\text{Cay}(G,S)$ is the graph with vertex set $G$ such that $x$ and $y$ are adjacent if and only if $x^{-1}y\in S$.  The chromatic number of the plane problem can be seen as asking for the chromatic number of a Cayley graph whose underlying group is $\mathbb{R}^2$ under addition, with the unit circle as the generating set.

In this paper, we propose the following variation on the chromatic number of the plane problem.  Let $n$ be a positive integer.  Let $\chi_{\text{max}}(n)$ denote the maximum chromatic number of $\text{Cay}(G,S)$ as $S$ ranges over all sets of the form $\{\pm v_1,\dots,\pm v_n\}$ where $v_1,\dots,v_n$ are (not necessarily distinct) unit vectors in $\mathbb{R}^2$, and $G$ is the subgroup of $\mathbb{R}^2$ generated by $S$.  We ask: What is $\chi_{\text{max}}(n)$?

Later in this introduction we state the primary main theorem of this paper.  One consequence of that theorem, which we regard as a secondary main result of this paper, is a computation of $\chi_{\text{max}}(n)$ for $n=1,2,3,4$.  Specifically, we prove the following:\begin{Cor}\label{cor-main}
    We have that $\chi_{\text{max}}(1)=\chi_{\text{max}}(2)=2$ and $\chi_{\text{max}}(3)=\chi_{\text{max}}(4)=3$.
\end{Cor}


Observe that every connected finite planar unit distance graph $X$ is isomorphic to a subgraph of one of our Cayley graphs $\text{Cay}(G,S)$.  To see that this is the case, let $e_1,\dots,e_n$ be the edges of $X$, and choose an arbitrary orientation of them.  Let $v_i$ be the terminal point minus the initial point of $e_i$.  Let $S=\{\pm v_1,\dots,\pm v_n\}$, and let $G$ be the subgroup of $\mathbb{R}^2$ generated by $S$.  Let $p_0$ be a vertex of $X$.  Then every vertex of $X$ is of the form $p_0+g$ for some $g\in G$; translating $X$ by $-p_0$ embeds $X$ in $\text{Cay}(G,S)$.  Provided one assumes the Axiom of Choice (AC), it therefore follows from the de Bruijn-Erd\H{o}s theorem \cite{de-Bruijn-Erdos} that $\chi(\mathbb{R}^2)=\chi_{\text{max}}(n)$ for sufficiently large values of $n$.  Without AC, set-theoretic issues may come into play; see \cite{Payne}.

We do not know the values of $\chi_{\text{max}}(n)$ for $n\geq 5$.  We have not yet sufficiently developed the techniques presented here for this purpose, but we hope to do so in future work.  We do get that $\chi_{\text{max}}(7)\geq 4$ by taking $v_1,v_2,v_3,v_4,v_5,v_6,v_7$ to be unit vectors defining the edges in the Mosers spindle \cite{Soifer}, a well-known $4$-chromatic planar unit distance graph.

\vspace{.1in}

The group $\mathbb{R}^2$ under addition is abelian, and so the above graphs $\text{Cay}(G,S)$ are \emph{abelian Cayley graphs}.  \cref{cor-main} will follow from our main theorem (\cref{thm:main}), a much more general result about chromatic numbers of connected abelian Cayley graphs with four generators.  We now introduce the notation needed to state that theorem.

Let $G$ be an abelian group, written additively, and let $S=\{\pm x_1,\dots,\pm x_m\}$ be a finite symmetric subset of $G$ such that $S$ generates $G$.  We say an $m$-tuple $(a_1,\dots,a_m)^t$ of integers (which we write as a column vector) is a \emph{relation} for $S$ if $a_1 x_1+\cdots a_m x_m=0$.  The set of relations for $S$ is a subgroup $H$ of $\mathbb{Z}^m$.  We say an $m\times r$ integer matrix $M$ is a \emph{Heuberger matrix} for $(G,S)$ if the columns of $M$ generate $H$.  As discussed in \cite{Cervantes_2023}, the matrix $M$ completely encodes all information about the graph $\text{Cay}(G,S)$.  (We remark that in the definition of Heuberger matrices, it would be more precise to refer to the $m$-tuple $(x_1,\dots,x_m)$ rather than the set $S$, because order matters here --- the $i$th row of $M$ corresponds to the $i$th generator $x_i$.  Technically we ought to pay no heed to the order in which elements of $S$ are listed, but for convenience we sometimes use set rather than tuple notation anyway, and we trust that this abuse of notation will not cause confusion.)

Conversely, given an $m\times r$ integer matrix $M$, we may construct an abelian group $G$ and a symmetric subset $S$ such that $M$ is a Heuberger matrix for $(G,S)$.  To wit, let $H$ be the subgroup of $\mathbb{Z}^m$ generated by the columns of $M$.  Let $G=\mathbb{Z}^m/H$, and let $S=\{H\pm e_1,\dots, H\pm e_m\}$, where $e_i$ is the column vector in $\mathbb{Z}^m$ with a $1$ in the $i$th position and $0$ elsewhere.  In this case we call $\text{Cay}(G,S)$ a \emph{standardized abelian Cayley graph} (or SACG for short), and we denote $\text{Cay}(G,S)$ by $M^{\text{SACG}}$.  If one begins with a pair $(G,S)$ and constructs $M$ as in the preceding paragraph, then indeed $\text{Cay}(G,S)$ is isomorphic to $M^{\text{SACG}}$.  Thus every connected finite-degree abelian Cayley graph is isomorphic to an SACG. 

The preprint \cite{Cervantes-small} establishes formulas for chromatic numbers of SACGs with small Heuberger matrices, namely, those of size $1\times r$, $m\times 1$, $2\times 2$, and $3\times 2$.  In the main theorem of this paper, we extend these results to determine the chromatic number of an SACG with a $4\times 2$ Heuberger matrix:

\begin{Thm}\label{thm:main}
    Let $M$ be a $4\times 2$ integer matrix, and let $X=M^\text{SACG}$. Suppose that $X$ is nonbipartite and has no loops, and that $M$ has no zero rows.  Then $\chi(X)=4$ if and only if there exists a signed permutation matrix $P$ and a unimodular matrix $U$ such that \[PMU=\begin{pmatrix}
        1 & a\\
        1 & b\\
        1 & c\\
        0 & 1 \end{pmatrix}\]for some integers $a,b,c$ with $3\,\mid\,a+b+c$, and $\chi(X)=3$ otherwise.
\end{Thm}

(Some notation and terminology used in the statement of \cref{thm:main}: A \emph{proper coloring} of a graph is an assignment of a color to each vertex such that if two vertices are adjacent, then they are assigned different colors.  The \emph{chromatic number} of a graph is the smallest number of colors with which it can be properly colored.  A \emph{loop} in a graph is an edge from a vertex to itself.  A graph with a loop cannot be properly colored.  SACGs are connected as well as regular of finite degree, and they always contain at least one edge; recall that such a graph is bipartite if and only if it has chromatic number $2$.  We say that a row of a matrix is a \emph{zero row} if every entry in it is $0$.  We write $\chi(X)$ for the chromatic number of a graph $X$.  A square integer matrix is a \emph{signed permutation matrix} if each row and each column has exactly one nonzero entry, which is $\pm 1$.  A \emph{unimodular} matrix is a square integer matrix with determinant $\pm 1$.)

In the course of our research, we did not make a beeline for \cref{cor-main}.  Rather, as we investigated chromatic numbers of abelian Cayley graphs with $4\times 2$ Heuberger matrices, we observed that it is a consequence of \cref{thm:main}.  Perhaps one can prove \cref{cor-main} much more quickly and directly than we do here.

\vspace{.1in}

We have organized this paper so that readers interested in unit distance graphs and willing to stipulate the correctness of \cref{thm:main} can stop after reading the first few sections.  In Section \ref{sec:initial-prelims}, we provide the background definitions and lemmas needed to get to \cref{cor-main}.  In Section \ref{sec:proof-of-cor}, we prove that \cref{cor-main} follows from \cref{thm:main}.  In Section \ref{sec:more-prelims}, we provide the additional background definitions and lemmas needed to get to \cref{thm:main}.  In Section \ref{sec:proof-main-theorem}, we prove \cref{thm:main}.

\section{Initial preliminaries}\label{sec:initial-prelims}

The background material presented in this section comes from \cite{Cervantes_2023}, where it is discussed in more detail.

\begin{Prop}[\cite{Cervantes_2023}]\label{thm:isomorphisms}
Let $M_X^\text{SACG}$ and $M_{X'}^\text{SACG}$ be standardized abelian Cayley graphs.\begin{enumerate}

\item If $M_{X'}$ is obtained by permuting the columns of $M_X$, then $X=X'$.

\item If $M_{X'}$ is obtained by multiplying a column of $M_X$ by $-1$, then $X=X'$.

\item\label{item-add-multiple of column} Suppose $y_j$ and $y_i$ are the $j$th and $i$th columns of $M_X$, respectively, with $j\neq i$.  If $M_{X'}$ is obtained by replacing the 
$j$th column of $M_X$ with $y_j+ay_i$ for some integer $a$, then $X=X'$.

\item If $M_{X'}$ is obtained by permuting the rows of $M_X$, then $X$ is isomorphic to $X'$.

\item If $M_{X'}$ is obtained by multiplying a row of $M_X$ by $-1$, then $X$ is isomorphic to $X'$.

\end{enumerate}
\end{Prop}

For (4), the isomorphism is obtained by permuting the generating set $\{H+e_1,\dots,H+e_m\}$.  The isomorphism from (5) is induced by $e_i\mapsto -e_i$.  The column operations in parts (1), (2), and (3) do not affect $H$, so we may use the same generating set $\{H+e_1,\dots,H+e_m\}$ for both, with elements in that order.

Using elementary linear algebra, one can show that $M_{X'}$ is obtained from $M_X$ by a finite sequence of column operations as in parts (1), (2), and (3) of \cref{thm:isomorphisms} if and only if there exists a unimodular matrix $U$ such that $M_{X'}=M_XU$.  Moreover, $M_{X'}$ is obtained from $M_X$ by a finite sequence of row operations as in parts (4) and (5) of \cref{thm:isomorphisms} if and only if there exists a signed permutation matrix $P$ such that $M_{X'}=PM_X$.

\vspace{.1in}

For any $m\times r$ integer matrix $M$ with $m\geq r$, there exists a unimodular matrix $U$ such that the top $r\times r$ minor of $MU$ is lower-triangular.  The procedure for finding $U$ is much like that of putting an integer matrix into Hermite normal form.

\begin{Prop}[\cite{Cervantes_2023}]\label{thm:delete-zero-columns}Let $M_X^\text{SACG}$ be a standardized abelian Cayley graph, where $M_X$ has a zero column.  Then $M_X^\text{SACG}=M_{X'}^\text{SACG}$, where $M_{X'}$ is obtained from $M_X$ by deleting a zero column.
\end{Prop}

Because of \cref{thm:delete-zero-columns}, we may always assume that the matrix $M_X$ is of full rank over $\mathbb{Q}$.  In particular, we may assume that the number of columns does not exceed the number of rows.

\vspace{.1in}

\begin{Prop}[\cite{Cervantes_2023}]\label{thm:delete-zero-rows}Let $M_X^\text{SACG}$ be a standardized abelian Cayley graph, where $M_X$ has a zero row.  Then $\chi(M_X^\text{SACG})=\chi(M_{X'}^\text{SACG})$, where $M_{X'}$ is obtained from $M_X$ by deleting a zero row.
\end{Prop}

Observe that \cref{thm:delete-zero-rows} does \emph{not} assert that the graphs in question are isomorphic, only that their chromatic numbers are equal.

\vspace{.1in}

For matrices consisting of a single column, the so-called Tomato Cage Theorem furnishes a formula for the chromatic number of the corresponding SACG:

\begin{Prop}[Tomato Cage Theorem, \cite{Cervantes_2023}]\label{thm:tomato-cage}Let $M^\text{SACG}$ be a standardized abelian Cayley graph, where $M$ consists of a single column.  If $M=\pm e_i$ for some $i$, then $M^\text{SACG}$ has loops and therefore cannot be properly colored.  Otherwise,\[\chi(M^\text{SACG})=\begin{cases}
    2 & \text{ if }s\text{ is even, and}\\
    3 & \text { if }s\text{ is odd}
\end{cases}\] where $s$ is the sum of the entries in $M$.\end{Prop}

\vspace{.1in}

Suppose $M$ is an $m\times r$ integer matrix of full rank over $\mathbb{Q}$.  Then $M^{\text{SACG}}$ has finitely many vertices if and only if $m=r$.

\section{Proof that \cref{thm:main} implies \cref{cor-main}}\label{sec:proof-of-cor}

In this section we prove \cref{cor-main}, under the assumption that \cref{thm:main} has already been proven.

\begin{proof}
    If $v$ is a unit vector in $\mathbb{R}^2$, and $S=\{\pm v\}$, and $G$ is the subgroup of $\mathbb{R}^2$ generated by $S$, then $\text{Cay}(G,S)$ is a $2$-regular tree, a.k.a. a doubly infinite path graph.  Hence $\chi_\text{max}(1)=2$.

    Now suppose that $v_1$ and $v_2$ are unit vectors in $\mathbb{R}^2$. Let $S=\{\pm v_1, v_2\}$, and let $G$ be the subgroup of $\mathbb{R}^2$ generated by $S$.  Then $\text{Cay}(G,S)$ is either an infinite grid graph or else a doubly infinite path graph, depending on whether $v_1=\pm v_2$.  Hence $\chi_\text{max}(2)=2$.

    For $n=3$ and $n=4$, we first establish a lower bound.  Take $v_1=(1,0), v_2=(-1/2,\sqrt{3}/2), v_3=(-1/2,-\sqrt{3}/2)$.  Let $S=\{\pm v_1, v_2, v_3\}$, and let $G$ (as always) be the subgroup of $\mathbb{R}^2$ generated by $S$.  The group $G$ can be visualized as a triangular lattice.  The graph $\text{Cay}(G,S)$ contains a $3$-cycle with vertices $(0,0)$, $(1,0)$, $(1/2,\sqrt{3}/2)$.  Hence $\chi_\text{max}(3)\geq 3$.  Taking $v_4=v_3$ and playing the same game all over again, we see that $\chi_\text{max}(4)\geq 3$.  In general, because vectors may be repeated, $\chi_\text{max}$ is a nondecreasing function.

    We now turn to the heart of this proof, in which we show that $\chi_\text{max}(4)\leq 3$.  To prove this, it will suffice to show that $\chi(\text{Cay}(G,S))\leq 3$ whenever $S=\{\pm v_1, \pm v_2, \pm v_3, \pm v_4\}$ for planar unit vectors $v_1, v_2, v_3. v_4$.  Once we show this, it will also prove that $\chi_\text{max}(3)\leq 3$ (because $\chi_\text{max}$ is nondecreasing), and so our work here will be done.

\vspace{.1in}

Let $M$ be a Heuberger matrix for $G$.  Note that $M$ has $4$ rows, and suppose that $M$ has $r$ columns.  As discussed in Section \ref{sec:initial-prelims}, we may assume that $M$ has full rank over $\mathbb{Q}$ (so that $r\leq 4$) and that the top $r\times r$ minor of $M$ is lower-triangular.  We divide into cases according to the value of $r$.  Let $X=\text{Cay}(G,S)$.

\guillemetright\, Suppose $r=1$.  Observe that $X$ does not have loops, because $v_i\neq 0$ for all $i$.  Therefore $\chi(X)\leq 3$ by \cref{thm:tomato-cage}.

\guillemetright\, Suppose $r=2$.  We seek to apply \cref{thm:main}.  As in the previous case, $X$ does not have loops.  If $X$ is bipartite, then $\chi(X)=2$, in which case we're done.  Now suppose that $M$ has a zero row, and let $M'$ be the matrix obtained by deleting that row.  Without loss of generality, suppose it is the fourth row of $M$ that was deleted.  Then $M'$ is a Heuberger matrix for $(G', S')$, where $S'=\{\pm v_1,\pm v_2, \pm v_3 \}$ and $G'$ is the subgroup of $\mathbb{R}^2$ generated by $S'$.  By \cref{thm:delete-zero-rows}, we have that $\chi(X)=\chi(X')$, where $X'=\text{Cay}(G',S')$.  Because $M$ has full rank and the top $2\times 2$ minor of $M$ is lower-triangular, it follows that the second column of $M'$ is of the form $(0,a,b)^t$ for some integers $a,b$, not both zero.  Hence $av_2+bv_3=0$.  But then $v_2$ and $v_3$ are linearly dependent unit vectors in $\mathbb{R}^2$, which occurs if and only if $v_2=\pm v_3$.  Thus $S'=\{\pm v_1,\pm v_2\}$, so $\chi(X')\leq 2$, because $\chi_\text{max}(2)=2$.

We now assume that $X$ is nonbipartite and has no loops, and that $M$ has no zero rows.  By \cref{thm:main}, we have that $\chi(X)=3$ unless there exists a signed permutation matrix $P$ and a unimodular matrix $U$ such that \[PMU=\begin{pmatrix}
        1 & a\\
        1 & b\\
        1 & c\\
        0 & 1 \end{pmatrix}.\]Suppose such $P$ and $U$ exist.  By \cref{thm:isomorphisms} and the remarks following it, $PMU$ is a Heuberger matrix for a graph $\tilde{X}$ isomorphic to $X$.  Moreover, $\tilde{X}=\text{Cay}(G,\tilde{S})$, where $\tilde{S}$ is obtained by permuting the elements of $S$ and negating some subset of them, in accordance with the action of $P$ on the rows of $M$.  (As we noted in the introduction, $\tilde{S}$ and $S$ should be regarded more precisely as $m$-tuples rather than as sets.)  Let $\tilde{S}=\{\pm\tilde{v_1},\pm\tilde{v_2},\pm\tilde{v_3},\pm\tilde{v_4}\}$.  For example, if \[P=\begin{pmatrix}
            0 & 1 & 0 & 0\\
            -1 & 0 & 0 & 0\\
            0 & 0 & 0 & 1\\
            0 & 0 & 1 & 0  \end{pmatrix},\]then $\tilde{v_1}=-v_2$ and $\tilde{v_2}=v_1$ and $\tilde{v_3}=v_4$ and $\tilde{v_4}=v_3$.

From the first column of $PMU$, then, we find that $\tilde{v_1}+\tilde{v_2}+\tilde{v_3}=0$.  It is then straightforward to show that these three unit vectors must define the sides of an equilateral triangle of side length $1$.  In other words, we must have that $\tilde{v_1}=R_\theta(1,0)$ and $\tilde{v_2}=R_\theta(-1/2,\sqrt{3}/2)$ and $\tilde{v_3}=R_\theta(-1/2,-\sqrt{3}/2)$, where $R_\theta$ is rotation by some fixed angle $\theta$.

From the second column of $PMU$, we find that $\tilde{v_4}=-a\tilde{v_1}-b\tilde{v_2}-c\tilde{v_3}$.  Thus $G$ is generated by $\tilde{v_1}, \tilde{v_2},$ and $\tilde{v_3}$.  So $G$ equals a triangular lattice in the plane, and $\tilde{v_4}$ is an element in this lattice.  There are precisely six unit vectors in this lattice, namely $\pm\tilde{v_1}, \pm\tilde{v_2},$ and $\pm\tilde{v_3}$.  Because $\tilde{v_4}$ is a unit vector, therefore $\tilde{v_4}=\pm\tilde{v_i}$ for some $i\in\{1,2,3\}$.  Thus $\tilde{X}=\text{Cay}(G,\{\pm\tilde{v_1}, \pm\tilde{v_2},\pm\tilde{v_3}\})$.  A proper coloring of $\tilde{X}$ is given by \[\alpha\tilde{v_1}+\beta\tilde{v_2}+\gamma\tilde{v_3}\mapsto \alpha+\beta+\gamma\;(\text{mod }3),\]so $\chi(\tilde{X})\leq 3$.  Because $X$ is isomorphic to $\tilde{X}$, this shows that $\chi(X)\leq 3$.

\guillemetright\,Suppose $r=3$.  Because $M$ has full rank and the top $3\times 3$ minor of $M$ is lower-triangular, it follows that the third column of $M$ is of the form $(0,0,a,b)^t$ for some integers $a,b$, not both zero.  Hence $av_3+bv_4=0$.  But then $v_3$ and $v_4$ are linearly dependent unit vectors in $\mathbb{R}^2$, which occurs if and only if $v_3=\pm v_4$.  Thus $X=\text{Cay}(G,S')$, where $S'=\{\pm v_1,\pm v_2,\pm v_3\}$.  Let $M'$ be a $3\times r'$ Heuberger matrix for $G,S'$.  Assume $M'$ has full rank (so $r'\leq 3$) and that the top $r'\times r'$ minor of $M'$ is lower-triangular.  Again we split into cases according to the value of $r'$.  If $r'=1$, then we have that $\chi(X)\leq 3$ by the Tomato Cage Theorem.  If $r'=2$, then use the ``lower-triangular argument'' once again to show that $v_2=\pm v_3$, which allows us to eliminate $v_3$.  Then $\chi(X)\leq\chi_\text{max}(2)=2$.  We cannot have $r'=3$, for then $X$ would be a finite graph.

\guillemetright\,Finally, we cannot have $r=4$, for then $X$ would contain only finitely many vertices.  But even a single unit vector generates an infinite subgroup of $\mathbb{R}^2$, so $X$ is an infinite graph\end{proof}

\section{More preliminaries}\label{sec:more-prelims}

In this section we provide more background, in particular the additional definitions and theorems needed for our proof of \cref{thm:main}.

Given two vertices $v$ and $w$ in a graph, we write $v\sim w$ to indicate that $v$ and $w$ are adjacent.  Recall that a \emph{graph homomorphism} is a function $f$ from the vertex set of one graph to the vertex set of another graph such that 
$f(v)\sim f(w)$ whenever $v\sim w$.  We write $X\xrightarrow{\ocirc}X'$ to indicate that there is a graph homomorphism from $X$ to $X'$.  Given two integer matrices $M_X$ and $M_{X'}$, we write $M_X \xrightarrow{\ocirc} M_{X'}$ to indicate that there is a graph homomorphism from $(M_X)^\text{SACG}$ to $(M_{X'})^\text{SACG}$.  If there is a graph homomorphism from $(M_X)^\text{SACG}$ to $(M_{X'})^\text{SACG}$, then $\chi((M_X)^\text{SACG})\leq\chi((M_X')^\text{SACG})$.

\begin{Prop}[\cite{Cervantes_2023}]\label{homomorphism-theorem}
Let $X$ and $X'$ be standardized abelian Cayley graphs with Heuberger matrices $M_X$ and $M_{X'}$, respectively. If $M_{X'}$ is obtained by ``collapsing'' any two rows of $M_X$ by adding them, then $M_X \xrightarrow{\ocirc} M_{X'}$.
\end{Prop}

The following two propositions offer direct methods for determining whether an SACG with a given Heuberger matrix contains loops or has chromatic number $2$, respectively.

\begin{Prop}[\cite{Cervantes_2023}]\label{prop:loops}
Let $X$ be an SACG with Heuberger matrix $M_X$. Then $X$ has loops if and only if a standard basis vector $e_j$ is in the $\Z$-span of the columns of $M_X$. 
\end{Prop}

\begin{Prop}[\cite{Cervantes_2023}]\label{prop:bipartite}
Let $X$ be an SACG with Heuberger matrix $M_X$. Then $X$ is bipartite if and only if the column sums of $M_X$ are all even.
\end{Prop}

We write $X\cong Y$ to indicate that a graph $X$ is isomorphic to a graph $Y$.

\begin{Ex}
    Let $X=M^{\text{SACG}}$, where \[M=\begin{pmatrix}
        9 & 21\\
        1 & 4
    \end{pmatrix}.\]Collapsing the two rows of $M$ by adding them, we find that $M\xrightarrow{\ocirc}(10\;\;\;25)$.  Applying \cref{thm:isomorphisms} to the latter matrix repeatedly, we find that \[(10\;\;\;25)^{\text{SACG}}\cong(10\;\;\;5)^{\text{SACG}}\cong(0\;\;\;5)^{\text{SACG}}\cong(5)^{\text{SACG}}.\]We have that $(5)^{\text{SACG}}$ is a $5$-cycle, so $\chi((10\;\;\;25)^{\text{SACG}})\leq 3$.  It follows that $\chi(X)\leq 3$.  By \cref{prop:bipartite}, we have that $X$ is not bipartite (because the sum of the second column of $M$ is $25$, which is odd), so $\chi(X)\geq 3$.  Therefore $\chi(X)=3$.
\end{Ex}

The next definition provides a standard form for a $3 \times 2$ Heuberger matrix. We then include a formula for determining the chromatic number of an arbitrary standardized abelian Cayley graph with a Heuberger matrix in this form. \cref{homomorphism-theorem} assures us that ``collapsing'' rows of a $4 \times 2$ matrix down to a $3 \times 2$ counterpart produces a graph homomorphism and therefore an upper bound on the chromatic number of the graph associated with the $4 \times 2$ matrix.  In this way we can take advantage of the formula for chromatic numbers of SACGs with $3\times 2$ Heuberger matrices.  This is our main tool for proving \cref{thm:main}.

\begin{Def}[modified Hermite Normal Form]\label{mHNF}
\normalfont Let $M_X$ be a $3 \times 2$ integer matrix with no zero rows or zero columns: 

$$
M_X = \begin{pmatrix}
    y_{11} & y_{12} \\
    y_{21} & y_{22} \\
    y_{31} & y_{32}
\end{pmatrix}_X
$$
\noindent We say $M_X$ is in \textit{modified Hermite Normal Form} if it meets the following conditions:

\begin{enumerate}
    \item $y_{11} > 0$, and
    \item $y_{12} = 0$, and
    \item $y_{11} \equiv 0 \textrm{ (mod 3)}$ or $y_{22} \equiv y_{32} \textrm{ (mod 3)}$, and
    \item $y_{22} \leq y_{32}$, and
    \item $|y_{22}| \leq |y_{32}|$, and
    \item If $y_{22} \neq 0$, then $-\frac{1}{2}|y_{22}| \leq y_{21} \leq 0$. Otherwise, $-\frac{1}{2}|y_{32}| \leq y_{31} \leq 0$.
    
\end{enumerate}
\end{Def}

\begin{Thm}[\cite{Cervantes-small}]\label{chi-of-mHNF}
Let $X$ be a standardized abelian Cayley graph with Heuberger matrix $M_X$. Suppose $M_X$ is in modified Hermite Normal Form:

$$
M_X = \begin{pmatrix}
    y_{11} & 0 \\
    y_{21} & y_{22} \\
    y_{31} & y_{32}
\end{pmatrix}_X
$$

\noindent Then we have the following:
\begin{enumerate}
    \item If the first column of $M_X$ is $e_1$ or the second column of $M_X$ is $e_3$, then $X$ has loops and therefore no proper coloring.
    \item If both of the column sums of $M_X$ are even, then $\chi(X) = 2$
    \item If $M_X$ belongs to one of the following six exceptional cases for some $k \in \N$ and $a \in \Z$ such that $3 \nmid a$, then $\chi(X)=4$.
    \begin{align}\notag
      (i) &  
      \begin{pmatrix}
        1 & 0 \\
        0 & 1 \\
        \pm 3k & 1 + 3k
    \end{pmatrix} &
      (ii) & 
      \begin{pmatrix}
        1 & 0 \\
        0 & -1 \\
        \pm 3k & -1 + 3k
    \end{pmatrix} & 
      (iii) &
      \begin{pmatrix}
        1 & 0 \\
        -1 & 2 \\
        -1 \pm 3k & 2 + 3k
    \end{pmatrix} 
    \end{align}
    \begin{align}\notag
    (iv) &
    \begin{pmatrix}
        1 & 0 \\
        -1 & -2 \\
        -1 \pm 3k & -2 + 3k
    \end{pmatrix} &
    (v) & 
    \begin{pmatrix}
        1 & 0 \\
        0 & -1 \\
        \pm3k & 2
    \end{pmatrix} &
    (vi) & 
    \begin{pmatrix}
        1 & 0 \\
        -1 & a \\
        -1 & a + 3(k-1)
    \end{pmatrix}
    \end{align}
    
    \item Otherwise, $\chi(X) = 3$
\end{enumerate}
\end{Thm}

Sometimes we are not easily able to get a $3 \times 2$ matrix into modified Hermite Normal Form. However, if we can get the matrix into a certain form, we can still get valuable information about the chromatic number.  The next two lemmas are variations on that theme.

\begin{Lem}[\cite{Tim}]\label{almost-mHNF-lemma}
Suppose the $3 \times 2$ matrix $M_X$ below has no rows or columns of all zeros, and satisfies $y_{22} \equiv y_{32} \textrm{ (mod 3)}$ with $y_{22} \neq y_{32}$. Suppose also that $\chi(X) = 4$, where $X=M_X^{\text{SACG}}$. Then we have that $|y_{22}| \in \{1,2 \}$, $|y_{32}| \in \{1,2 \}$, or $|y_{21}| = 1$.

$$
M_X = \begin{pmatrix}
    1 & 0 \\
    y_{21} & y_{22} \\
    y_{21} & y_{32}
\end{pmatrix}
$$

\end{Lem}

\begin{Lem}[Three-Divisible Row Lemma, \cite{Tim}]\label{tim-3k-lemma}
For $k_1,k_2 \in \Z$, let $X=M_X^{\text{SACG}}$, where$$
M_X = \begin{pmatrix}
    3k_1 & 3k_2 \\
    y_{21} & y_{22} \\
    y_{31} & y_{32}
\end{pmatrix}.
$$

\noindent Suppose that $M_X$ has no zero rows and $X$ has no loops. Then $\chi(X) \leq 3$. 
\end{Lem}

The next lemma shows that an SACG with a $4 \times 2$ Heuberger matrix of a particular form has a chromatic number of four.

\begin{Lem}[Tri-Triangle Lemma, \cite{Tim}]\label{tri-triangle}
Let $X=M_X^{\text{SACG}}$, where
$$
M_X = \begin{pmatrix}
    1 & y_{12} \\
    1 & y_{22} \\
    1 & y_{32} \\
    0 & y_{42} \\
\end{pmatrix}.
$$

\noindent Suppose that $X$ does not have loops. Then: 

$$
\chi(X) = \begin{cases}
  4 & \textit{if } \ |y_{42}| = 1 \textrm{ \textit{and} } 3 \ | \ y_{12}+y_{22}+y_{32}\\
  3 & otherwise.
  \end{cases}
$$
\end{Lem}

In \cite{Tim}, a system is established for categorizing most $4 \times 2$ Heuberger matrices into one of three ``buckets.''

\begin{Def}[pre-modified Hermite Normal Form]\label{def:pMHNF}
\normalfont Let $M$ be a $4 \times 2$ integer matrix.  Also let $A_{1,2}, \ A_{2,3}$, and $A_{3,4}$ be defined as follows:
\begin{align}\notag
A_{1,2} &= 
\begin{pmatrix}
1 & 1 & 0 & 0\\
0 & 0 & 1 & 0\\
0 & 0 & 0 & 1
\end{pmatrix} 
&
A_{2,3} &= \begin{pmatrix}
1 & 0 & 0 & 0\\
0 & 1 & 1 & 0\\
0 & 0 & 0 & 1
\end{pmatrix}
&
A_{3,4} &= \begin{pmatrix}
1 & 0 & 0 & 0\\
0 & 0 & 1 & 0\\
0 & 0 & 1 & 1
\end{pmatrix}
\end{align}

We say that $M$ is in \textit{pre-modified Hermite Normal Form} if 

\begin{enumerate}
    \item There exists an $i \in \{1,2,3\}$ such that $A_{i,i+1}M_X$ either has a zero row or is in modified Hermite Normal Form, and
    \item If $i=1$, then every entry of the top row of $A_{1,2}M_X$ is divisible by $3$.
\end{enumerate}

\end{Def}

Observe that $A_{i,i+1}M$ is the matrix obtained by collapsing rows $i$ and $i+1$ of $M$ by adding them.

\begin{Prop}\cite{Tim}
Let $M_X$ be a $4 \times 2$ integer matrix with no zero rows and no zero columns, where $X=M_X^\text{SACG}$ itself has no loops. Then, there exists a $4 \times 2$ integer matrix $M_Y$ in pre-modified Hermite Normal Form and a signed permutation matrix $P$ and a unimodular matrix $U$ such that $PM_XU=M_Y$.
\end{Prop}

In our proof of \cref{thm:main}, then, we will begin by dividing into three cases, according the minimum value of $i$ for which our matrix satisfies \cref{def:pMHNF}.

\section{Proof of \cref{thm:main}}\label{sec:proof-main-theorem}

Part of the proof of \cref{thm:main} can be found in the series \cite{Tim, Kathy, Mason} of master's theses.  In particular, \cite{Tim} proves that \cref{thm:main} holds for every $4\times 2$ integer matrix in Case 1, while \cite{Kathy} and \cite{Mason} prove that it holds for many subcases of Cases 2 and 3.  In this section, we describe how these cases split into subcases as well as how to prove the theorem for the remaining subcases.

First, though, we prove a lemma that will be useful in many of the subcases.  This can be seen as a version of \cref{lem:4-by-2-three-div-row-lemma} for $4\times 2$ matrices.

\begin{Lem}[$4\times 2$ three-divisible row lemma]\label{lem:4-by-2-three-div-row-lemma}
    Let $M_Y$ be a $4\times 2$ integer matrix with associated SACG $Y$ such that $M_Y$ has no zero rows and $Y$ has no loops.  Suppose that some row of $M_Y$ is divisible by $3$.  Then $\chi(Y)\leq 3$.
\end{Lem}

\begin{proof}
    Write \[M_Y=\begin{pmatrix}
        y_{11} & y_{12}\\
        y_{21} & y_{22}\\
        y_{31} & y_{32}\\
        y_{41} & y_{42}
    \end{pmatrix}.\]Assume without loss of generality that $3$ divides the top row of $M_Y$, that is, $3\,\mid\,y_{11}$ and $3\,\mid\,y_{12}$.

By \cref{homomorphism-theorem} we have that \[(M_Y)^\text{SACG}\xrightarrow{\ocirc}\begin{pmatrix}
    y_{11} & y_{12}\\
    y_{21} & y_{22}\\
    y_{31}+y_{41} & y_{32}+y_{42}\\
\end{pmatrix}^\text{SACG}=Z_1,\]where $M_{Z_1}$ is the latter matrix.  For sake of contradiction, assume that the statement ``$\chi(Y)\leq 3$'' fails.  (This includes the possibility that $Y$ has loops.)

By \cref{tim-3k-lemma}, either $M_{Z_1}$ has a zero row or $Z_1$ has loops.

Let's first say that $M_{Z_1}$ has a zero row.  Because $M_Y$ has no zero rows, it must be that the third row of $M_{Z_1}$ is a zero row.  That is, $y_{41}=-y_{31}$ and $y_{42}=-y_{32}$.  By \cref{thm:isomorphisms} and \cref{homomorphism-theorem} (namely, multiply the bottom row by $-1$, then merge the bottom two rows by adding), we have a homomorphism\[(M_Y)^\text{SACG}\xrightarrow{\ocirc}\begin{pmatrix}
    y_{11} & y_{12}\\
    y_{21} & y_{22}\\
    2y_{31} & 2y_{32}\\
\end{pmatrix}^\text{SACG}=Z_2,\]where $M_{Z_2}$ is the latter matrix.  Again by \cref{tim-3k-lemma}, we get that $M_{Z_2}$ has a zero row, or else $Z_2$ has loops.  However, $M_{Z_2}$ cannot have a zero row, else $M_Y$ would have a zero row.  So $Z_2$ has loops.  By \cref{prop:loops}, therefore $\alpha c_1+\beta c_2=e_i$, where $\alpha, \beta\in\mathbb{Z}$ and $c_j$ is the $j$th column of $M_{Z_2}$ and $i\in\{1,2,3\}$.  Note that the bottom row of $M_{Z_2}$ is divisible by $2$, so we cannot have $i=3$.  But then $\alpha d_1+\beta d_2=e_i$, where $d_j$ is the $j$th column of $M_Y$.  So $Y$ has loops, which is a contradiction.

\vspace{.1in}

Hence we have established that $Z_1$ has loops.  By \cref{prop:loops}, therefore \begin{equation}\label{eq:Z-1}
    \alpha_1\begin{pmatrix}
        y_{11}\\
        y_{21}\\
        y_{31}+y_{41}
\end{pmatrix}+\beta_1\begin{pmatrix}
        y_{12}\\
        y_{22}\\
        y_{32}+y_{42}
    \end{pmatrix}=e_{i_1}
\end{equation}for some $i_1\in\{2,3\}$ and for some $\alpha_1,\beta_1\in\mathbb{Z}$.  (Note that we cannot have $i_1=1$, by our assumption that $3$ divides $y_{11}$ and $y_{12}$.)

Playing the same game by subtracting rows instead of adding them, we find that we have a homomorphism\[(M_Y)^\text{SACG}\xrightarrow{\ocirc}\begin{pmatrix}
    y_{11} & y_{12}\\
    y_{21} & y_{22}\\
    y_{31}-y_{41} & y_{32}-y_{42}\\
\end{pmatrix}^\text{SACG}=Z_2,\]where $M_{Z_2}$ is the latter matrix.  Along the same lines of our earlier argument, we find that $M_{Z_2}$ has no zero rows, and thus by \cref{tim-3k-lemma} and \cref{prop:loops}, we have that \begin{equation}\label{eq:Z-2}
    \alpha_2\begin{pmatrix}
        y_{11}\\
        y_{21}\\
        y_{31}-y_{41}
\end{pmatrix}+\beta_2\begin{pmatrix}
        y_{12}\\
        y_{22}\\
        y_{32}-y_{42}
    \end{pmatrix}=e_{i_2}
\end{equation}for some $i_2\in\{2,3\}$ and for some $\alpha_2,\beta_2\in\mathbb{Z}$.

Because $M_Y$ has no zero rows, therefore $(y_{11}, y_{12})$ is not the zero vector.  Equations (\ref{eq:Z-1}) and (\ref{eq:Z-2}) then imply that $(\alpha_1,\beta_1)$ and $(\alpha_2,\beta_2)$ both lie in the (one-dimensional) orthogonal complement of $(y_{11}, y_{12})$ in $\mathbb{Q}^2$.  So \begin{equation}\label{eq:alphas-betas}
    \alpha_1\beta_2=\alpha_2\beta_1.
\end{equation}

Equation (\ref{eq:Z-1}) shows that some linear combination of $\alpha_1$ and $\beta_1$ with integer coefficents equals 1.  Likewise with Equation (\ref{eq:Z-2}).  So $\text{gcd}(\alpha_1,\beta_1)=\text{gcd}(\alpha_2,\beta_2)=1$.

From this and (\ref{eq:alphas-betas}) we get that either $\alpha_1=\alpha_2$ and $\beta_1=\beta_2$, or else $\alpha_1=-\alpha_2$ and $\beta_1=-\beta_2$.

Suppose that $\alpha_1=\alpha_2$ and $\beta_1=\beta_2$.  (The case where $\alpha_1=-\alpha_2$ and $\beta_1=-\beta_2$ will be similar and hence is omitted.)  We divvy now into cases according to the values of $i_1$ and $i_2$.

\guillemetright\, Suppose $i_1=i_2=2$.  Adding (\ref{eq:Z-1}) and (\ref{eq:Z-2}) gives us that $\alpha_1 y_{31}+\beta_1 y_{32}=0$.  Subtracting (\ref{eq:Z-1}) and (\ref{eq:Z-2}) gives us that $\alpha_1 y_{41}+\beta_1 y_{42}=0$.  With (\ref{eq:Z-1}), then, that gives us that $\alpha_1 d_1+\beta_1 d_2=e_2$, where $d_j$ is the $j$th column of $M_Y$.  But then $Y$ has loops, which is a contradiction.

\guillemetright\, Suppose that either $i_1=2$ and $i_2=3$, or else $i_1=3$ and $i_2=2$.  Adding (\ref{eq:Z-1}) and (\ref{eq:Z-2}) gives us that $2\alpha_1 y_{31}+2\beta_1 y_{32}=1$, which is a contradiction.

\guillemetright\, Suppose $i_1=i_2=3$.  Adding (\ref{eq:Z-1}) and (\ref{eq:Z-2}) gives us that $\alpha_1 y_{31}+\beta_1 y_{32}=1$.  Subtracting (\ref{eq:Z-1}) and (\ref{eq:Z-2}) gives us that $\alpha_1 y_{41}+\beta_1 y_{42}=0$.  With (\ref{eq:Z-1}), then, that gives us that $\alpha_1 d_1+\beta_1 d_2=e_3$, where $d_j$ is the $j$th column of $M_Y$.  But then $Y$ has loops, which is a contradiction.\end{proof}

In Cases 2 and 3, therefore, we can assume that no row of our original $4\times 2$ matrix $M_Y$ is divisible by $3$.  Moreoever, we can assume that there does not exist a pair of rows in $M_Y$ whose sum or difference is divisible by $3$; otherwise, we would be in Case 1.

If we are in Case 2, let \[M_Z=\begin{pmatrix}
    y_{11} & y_{12}\\
    y_{21}+y_{31} & y_{22}+y_{32}\\
    y_{41} & y_{42}
\end{pmatrix}.\]If we are in Case 3, let \[M_Z=\begin{pmatrix}
    y_{11} & y_{12}\\
    y_{21} & y_{22}\\
    y_{31}+y_{41} & y_{32}+y_{42}
\end{pmatrix}.\]

In either case, we have that $M_Z$ is in modified Hermitian normal form.  In either case, let $Z=(M_Z)^\text{SACG}$.  To prove \cref{thm:main}, we assume $\chi(Y)\geq 4$.  Because there is a graph homomorphism $M_Y\xrightarrow{\ocirc} M_Z$, therefore $\chi(Z)\geq 4$.  By \cref{chi-of-mHNF}, we have that either $Z$ has loops, or else $M_Z$ equals one of the six exceptional cases in that theorem with chromatic number $4$.

We can reduce the number of cases to consider.  Observe that each of the first four exceptional cases can be tranformed into the sixth by appropriate row and column operations.  For example, for the matrix \[\begin{pmatrix}
    1 & 0\\
    0 & 1\\
    3k & 1+3k
\end{pmatrix},\]which corresponds to (i), we find that adding $-1$ times the second column to the first column produces a new matrix in the form (vi).  Such row and column operations can be ``absorbed'' into the matrices $P$ and $U$ we seek to find.  Thus, we need not consider exceptional cases (i), (ii), (iii), and (iv).

Moreover, by \cref{chi-of-mHNF}(1), we have that $Z$ has loops if and only if the first column of $M_Z$ is $e_1$, or the second column is $e_3$.  We claim that the latter possibility cannot occur.  For suppose otherwise.  The top row of $M_Z$ coincides with the top row of $M_Y$ and hence is not divisible by $3$.  The top right entry of $M_Z$ is $0$, because it is the first entry in $e_3$.  So the top left entry of $M_Z$ is not divisible by $3$.  From \cref{mHNF}, it follows that the middle left and bottom left entries of $M_Z$ are congruent modulo $3$.  Hence the second column of $M_Z$ cannot be $e_3$.

So far we have reduced down to six sub-cases: For each of Cases 2 and 3, we consider the possibilities where the first column of $M_Z$ is $e_1$; where $M_Z$ is in exceptional case (v) from \cref{chi-of-mHNF}; and where $M_Z$ is in exceptional case (vi) from \cref{chi-of-mHNF}.

We can consolidate (indeed, eliminate) cases further still.  For the situation where we are in Case 3 and the first column of $M_Z$ is $e_1$, \cref{thm:main} is proven in \cite{Mason}.  A careful review of that proof shows that conditions (4), (5), and (6) from \cref{mHNF} were never used there.  The remaining conditions are symmetric with respect to transposition of the bottom two rows.  Hence if we are in the situation where we are in Case 2 and the first column of $M_Z$ is $e_1$, this can be transformed into the same situation for Case 3 by swapping rows 2 and 4 of $M_Y$.  This row swap corresponds to multiplication on the left by a permutation matrix, which can be ``absorbed'' into the matrix $P$ from \cref{thm:main}.

A similar line of reasoning shows that exceptional case (vi) from \cref{chi-of-mHNF} does not have to be considered separately for Cases 2 and 3.  That's because in our proof of \cref{thm:main} for the situation where we are in Case 3 and $M_Z$ is in exceptional case (vi) from \cref{chi-of-mHNF}, at no point do we use that $k>0$.  Thus, as in the previous paragraph, we may swap rows $2$ and $4$ of $M_Y$ to get from Case 2 to Case 3.

There are, then, only three sub-cases remaining to consider: Case 2, when $M_Z$ is in exceptional case (v) from \cref{chi-of-mHNF}; Case 3, when $M_Z$ is in exceptional case (v) from \cref{chi-of-mHNF}; and Case 3, when $M_Z$ is in exceptional case (vi) from \cref{chi-of-mHNF}.  In the last of these, we must be careful to avoid using the information that $k>0$.

The techniques used to prove \cref{thm:main} for these three sub-cases are quite similar for each, and to present all of them would be tedious in the extreme.  Consequently, we present only one.  Armed with these methods, the reader then can prove the remaining ones, if desired.


\section{Case 3, when $M_Z$ is in exceptional case (v)}

\[
\begin{pmatrix}
    1 & 0 \\
    0 & -1 \\
    y_{31} & y_{32} \\
    y_{41} & 2- y_{32}
\end{pmatrix}_{Y} 
\xrightarrow{\ocirc}
\begin{pmatrix}
    1 & 0 \\
    0 & -1 \\
    \pm3k & 2
\end{pmatrix}_{Z} 
\]

We can start by considering the congruences of each of the variables modulo 3. There are many restrictions that apply here, so for some $\epsilon = \pm1$, we find that the only possible reduction of $Y$ modulo 3 is as follows: 

\[
\begin{pmatrix}
    1 & 0 \\
    0 & -1 \\
    \epsilon & 1 \\
    -\epsilon & 1
\end{pmatrix}_{Y \textrm{ mod 3}} 
\cong
\begin{pmatrix}
    1 & 0 \\
    0 & -1 \\
    1 & 1 \\
    -1 & 1
\end{pmatrix}_{Y \textrm{ mod 3}}
\]

Now we can safely assume the congruence of each variable modulo 3. This allows us to construct new homomorphisms that can be used to verify \cref{thm:main}. Here is one potential homomorphism to start: 

\[
\begin{pmatrix}
    1 & 0 \\
    0 & -1 \\
    y_{31} & y_{32} \\
    y_{41} & 2- y_{32}
\end{pmatrix}_{Y} 
\xrightarrow{\ocirc}
\begin{pmatrix}
    1 & 0 \\
    -y_{31} & -1-y_{32} \\
    y_{41} & 2-y_{32}
\end{pmatrix}_{U} 
\]

This $M_U$ was constructed in such a way as to ensure that the first three conditions of modified Hermite Normal Form are met. When these conditions are met, there is a standard methodology that can be applied to find the cases where $\chi(U)=4$. First though, we'll consider the possibility that $U$ has loops by applying \cref{prop:loops}. In other words, we seek solutions to this system of equations:

$$\alpha \begin{pmatrix}
1 \\
-y_{31}  \\
y_{41} 
\end{pmatrix} + \beta \begin{pmatrix}
0 \\
-1-y_{32}  \\
2-y_{32} 
\end{pmatrix} = e_i$$

for some $i\in\{1,2,3\}$. Considering the congruence of each of the variables, we are able to immediately eliminate all possibilities except for the $i=1$ case. Here is one equation we can pull from that system: $$3 y_{41}-(y_{31}+y_{41})(2-y_{32})=0$$

We'll come back to this, but for now, we move on to identifying the cases where $\chi(U)=4$. \\

In order to do this, we will attempt to put $M_U$ into modified Hermite Normal Form. We can ensure the fourth and fifth conditions are met by multiplying the second column by $\pm 1$ and/or swapping rows. To ensure the sixth condition is met, we can subtract an integer multiple of the second column from the first and multiply the first row and column by $\pm 1$. Therefore for some $q \in \Z$ and $\epsilon = \pm1$, one of the following two matrices (each of which corresponds to a graph isomorphic to $U$) is in modified Hermite Normal Form:   

\[
\begin{pmatrix}
    1 & 0 \\
    \epsilon[-y_{31}-q(-1-y_{32})] & -1-y_{32} \\
    \epsilon[y_{41}-q(2-y_{32})] & 2-y_{32}
\end{pmatrix}_{U_1} 
\textrm{ or  } \
\begin{pmatrix}
    1 & 0 \\
    \epsilon[y_{41}-q(-2+y_{32})] & -2+y_{32} \\
    \epsilon[-y_{31}-q(1+y_{32})] & 1+y_{32}
\end{pmatrix}_{U_2}
\]

Comparing $M_{U_1}$ and $M_{U_2}$ to the six exceptional cases from \cref{chi-of-mHNF} will give all possible sub-cases where $\chi(U)=4$. We are able to quickly rule out most possibilities by finding some contradiction with the congruence of the variables, or by finding that $M_Y$ has a tri-triangle (for example, when $y_{32}=1$). The sub-cases that are not so quickly addressed are as follows: 

\begin{enumerate}
    \item $U$ has loops: \\
    $3 y_{41}-(y_{31}+y_{41})(2-y_{32})=0$ 
    \item $U_1$ belongs to case (i): \\
    $y_{32} = -2$ and $y_{41} = \pm3-4y_{31}$ 
    \item $U_1$ or $U_2$ belong to case (vi): \\
    $3y_{41}-(y_{31}+y_{41})(2-y_{32})=\pm3$
    \item $U_2$ belongs to case (iii): \\
    $y_{32}=4$
\end{enumerate}

There are a number of approaches we can take to address the remaining sub-cases. One option is to construct another homomorphism, repeat this process, and then consider the cases where \textit{both} homomorphisms fail to quickly give an upper bound for $\chi(Y)$. We will proceed with this method and call the new matrix $M_V$:

\[
\begin{pmatrix}
    1 & 0 \\
    0 & -1 \\
    y_{31} & y_{32} \\
    y_{41} & 2- y_{32}
\end{pmatrix}_{Y} 
\xrightarrow{\ocirc}
\begin{pmatrix}
    1 & 0 \\
    0 & -1 \\
    y_{31}-y_{41} & 2y_{32}-2
\end{pmatrix}
\cong\begin{pmatrix}
    1 & 0 \\
    0 & -\epsilon \\
    2y_{32}-2 & \epsilon (y_{31}-y_{41})
\end{pmatrix}_{V} 
\]

Here we are again letting $\epsilon=\pm1$. The process is much simpler for $M_V$ since the entries in the middle row are essentially known. We are able to quickly rule out the loops case and find only two sub-cases that can not be quickly addressed: 

\begin{enumerate}
    \item $M_V$ belongs to case (i) or (ii): \\
    $- y_{31}+y_{41}-1=\pm(2 y_{32}-2)$
    \item $M_V$ belongs to case (v): \\
    $y_{41}=y_{31}-2$
\end{enumerate}

We can now cross-reference our two lists of sub-cases (which gives a total of eight possibilities) and verify \cref{thm:main} applies for each of them individually. In general, we try to avoid proliferating our sub-cases.  However, here we find the reduction in complexity to be worth the increased number of cases here. Most of these new sub-cases reduce to the point of being trivial.  We will address a few of the more complex ones. \\

\textbf{$U$ has loops, and $M_V$ belongs to case (i) or (ii):}

We can combine the two associated equations to find: $$6y_{41}-(y_{31}+y_{41})(2\pm2(-y_{31}+y_{41}-1))=0$$

The unknown sign in the equation gives two possible equations that can be evaluated. Using the reduction theory of quadratic forms as in \cite{Jarvis}, we will find the integer solutions to these equations and list only solutions that are consistent with the variables' congruence modulo three:

\begin{center}
    \begin{tabular}{c|c}
    Equation & Integer Solutions ($y_{31},y_{41}$) \\ \hline
    $2y_{31}^2+6y_{41}-2y_{41}^2=0$ & $(-2,-1)$ \\
    $-y_{31}^2-3y_{31}+3 y_{41}+y_{41}^2=0$ & $y_{31}=-y_{41}-3$
\end{tabular}
\end{center}


If we have $y_{31}=-2$ and $y_{41}=-1$, every entry is known and it's straightforward to verify the result of \cref{thm:main}. If we have $y_{31}=-y_{41}-3$, we can substitute this into the previous system of equations (since we should still be in the case where $U$ has loops) and verify that the system has no solutions. So we have a contradiction and are done with this case. The sub-case where $U$ has loops and $M_V$ belongs to case (v) is similar. \\

\textbf{$U_2$ belongs to case (iii):} \\

Here we find that $M_V$ is a little difficult to incorporate, so we'll construct a new homomorphism entirely: 

\[
\begin{pmatrix}
    1 & 0 \\
    0 & -1 \\
    y_{31} & 4 \\
    y_{41} & -2 
\end{pmatrix}_{Y} 
\xrightarrow{\ocirc}
\begin{pmatrix}
    1 & 0 \\
    4 & y_{31} \\
    -2 & y_{41}-1
\end{pmatrix}_{W} 
\]

We can follow the same process we used for $M_U$ to address the sub-cases where $W$ has loops or where $\chi(W)=4$. This involves solving a number of relatively straightforward Diophantine equations. One additional tool that is helpful here is noting that sub-cases where $y_{31}=-y_{41}$ can be immediately disregarded. This is because $y_{31}=-y_{41}$ implies that $Z$ has loops, and that case was previously evaluated in \cite{Mason}. The reader can verify this result, as well as the remaining sub-cases. 

\bibliographystyle{amsplain}
\bibliography{4-by-2}

\providecommand{\bysame}{\leavevmode\hbox to3em{\hrulefill}\thinspace}
\providecommand{\MR}{\relax\ifhmode\unskip\space\fi MR }
\providecommand{\MRhref}[2]{%
  \href{http://www.ams.org/mathscinet-getitem?mr=#1}{#2}
}
\providecommand{\href}[2]{#2}
\begin{thebibliography}{1}

\bibitem{Cervantes_2023}
Jonathan Cervantes and Mike Krebs, \emph{Chromatic numbers of {C}ayley graphs
  of abelian groups: A matrix method}, Linear Algebra and its Applications
  \textbf{676} (2023), 277–295.

\bibitem{Cervantes-small}
Jonathan Cervantes and Mike Krebs, \emph{Chromatic numbers of {C}ayley graphs
  of abelian groups: Cases of small dimension and rank}, 2023.

\bibitem{de-Bruijn-Erdos}
N.~G. de~Bruijn and P.~Erd\H{o}s, \emph{A colour problem for infinite graphs
  and a problem in the theory of relations}, Nederl. Akad. Wetensch. Proc. Ser.
  A. {\bf 54} = Indagationes Math. \textbf{13} (1951), 369--373. \MR{0046630}

\bibitem{Tim}
Timothy Harris, \emph{Standardized abelian {C}ayley graphs and their chromatic
  numbers in four dimensions}, M.{S}. thesis, California State University, Los
  Angeles, 2024, https://www.calstatela.edu/faculty/mike-krebs.

\bibitem{Jarvis}
Frazer Jarvis, \emph{Algebraic number theory}, Springer Undergraduate
  Mathematics Series, Springer, Cham, 2014. \MR{3289993}

\bibitem{Mason}
Mason Meeks, \emph{Standardized abelian {C}ayley graphs and their chromatic
  numbers in four dimensions: A partial evaluation of case three}, M.{S}.
  thesis, California State University, Los Angeles, 2025,
  https://www.calstatela.edu/faculty/mike-krebs.

\bibitem{Kathy}
Katherine Ortiz, \emph{Standardized abelian {C}ayley graphs and their chromatic
  numbers in four dimensions: Case two}, M.{S}. thesis, California State
  University, Los Angeles, 2025, https://www.calstatela.edu/faculty/mike-krebs.

\bibitem{Payne}
Michael~S. Payne, \emph{Unit distance graphs with ambiguous chromatic number},
  Electron. J. Combin. \textbf{16} (2009), no.~1, Research Paper N31.

\bibitem{Soifer}
Alexander Soifer, \emph{The new mathematical coloring book: Mathematics of
  coloring and the colorful life of its creators}, 01 2024.

\end{thebibliography}

\end{document}